\documentclass[12pt]{amsart}
\usepackage{amssymb,latexsym,eufrak,amsmath,amscd, graphicx}
\usepackage[colorlinks=true,pagebackref,hyperindex]{hyperref}

\usepackage{amsfonts}
\usepackage{amscd}

\usepackage{xypic}

\renewcommand{\phi}{\varphi}
\renewcommand{\epsilon}{\varepsilon}
\renewcommand{\theta}{\vartheta}

\def\RR{{\mathbf R}}
\def\QQ{{\mathbf Q}}

\def\BB{{\mathbf B}}

\def\cJ{\mathcal{J}}

\def\cF{\mathcal{F}}

\def\cO{\mathcal{O}}

\def\fra{\mathfrak{a}}
\def\frb{\mathfrak{b}}
\def\frc{\mathfrak{c}}

\def\frm{\mathfrak{m}}

\DeclareMathOperator{\ord}{ord}

\DeclareMathOperator{\NS}{N}

\newtheorem{lemma}{Lemma}[section]
\newtheorem{theorem}[lemma]{Theorem}
\newtheorem{corollary}[lemma]{Corollary}
\newtheorem{proposition}[lemma]{Proposition}

\theoremstyle{definition}

\newtheorem{remark}[lemma]{Remark}

\theoremstyle{remark}
\newtheorem*{remark*}{Remark}
\newtheorem*{note*}{Note}

\begin{document}

\title{The non-nef locus in positive characteristic}

\thanks{2010\,\emph{Mathematics Subject Classification}.
 Primary 14A10; Secondary 13A35.
\newline The author was partially supported by
 NSF grant DMS-1068190 and
  a Packard Fellowship.}

\author[M.~Musta\c{t}\u{a}]{Mircea~Musta\c{t}\u{a}}
\address{Department of Mathematics, University of Michigan,
Ann Arbor, MI 48109, USA}
\email{{mmustata@umich.edu}}

\dedicatory{Dedicated to Joe Harris on the~occasion of
his~sixtieth~birthday}

\baselineskip 16pt \footskip = 32pt

\begin{abstract}
We give an analogue in positive characteristic of 
the description of the non-nef locus from \cite{ELMNP}. In this case, the role of the asymptotic multiplier ideals is played by the asymptotic test ideals. The key ingredient is provided by a uniform
global generation statement involving twists by such ideals.
\end{abstract}

\maketitle

\markboth{M.~MUSTA\c{T}\u{A}}{THE NON-NEF LOCUS IN POSITIVE CHARACTERISTIC}

\section{Introduction}

Let $X$ be a smooth, projective variety over an algebraically closed field $k$. If $D$
is a divisor on $X$, then for every positive integer $m$ we may consider 
the closed subset ${\rm Bs}(mD)$, the base-locus
of the linear system $|mD|$.
The intersection $\bigcap_{m\geq 1}{\rm Bs}(mD)$ is equal to ${\rm Bs}(\ell D)$ for $\ell$ divisible enough;
this is the \emph{stable base locus} $\BB(D)$ of $D$. By definition, we have 
$\BB(D)=\BB(rD)$ for every positive integer $r$, and using this one extends in the obvious 
way the definition of $\BB(D)$ to the case when $D$ is a $\QQ$-divisor.

The stable base locus is a very interesting invariant, but it is quite subtle: for example, two
numerically equivalent divisors can have different stable base loci.
A related subset is the \emph{non-nef locus}, defined as follows. If $D$ is an $\RR$-divisor
on $X$, then
$$\BB_-(D):=\bigcup_A\BB(D+A),$$
where the union is over all ample $\RR$-divisors $A$ such that $D+A$ is a $\QQ$-divisor. 
It follows from the definition that $\BB_-(D)$ only depends on the numerical equivalence class 
of $D$, and $\BB_-(D)$ is empty if and only if $D$ is nef. 

This locus was studied in \cite{ELMNP} over a ground field of characteristic zero. 
The key tool in this study is the asymptotic multiplier ideal and a certain uniform global generation result for twists by such ideals. 
In that context, the global generation statement is a consequence of vanishing theorems
of Kodaira-type and of Castelnuovo-Mumford regularity.
The main point of the present paper is that
a similar uniform global generation result also holds in positive characteristic, if one replaces
the asymptotic multiplier ideal by the asymptotic test ideal (despite the fact that in positive characteristic Kodaira's vanishing theorem and its generalizations may fail).

We recall that test ideals give an analogue in positive characteristic of  multiplier ideals
in characteristic zero. They were introduced by Hara and Yoshida in
\cite{HY} using a generalization of tight closure theory, and it was noticed from the beginning
that they satisfy similar formal properties with those of multiplier ideals in characteristic zero.
Furthermore, there are some very interesting results and open problems concerning the connection between multiplier ideals and test ideals via reduction mod $p$.
We refer to \S 3 for the definition of test ideals and to
\cite{ST} for a more comprehensive overview. 

As it is the case for multiplier ideals in characteristic zero, for a divisor $D$ on a variety $X$
in positive characteristic with $\cO_X(D)$ of non-negative Iitaka dimension, 
one can use an asymptotic construction
to obtain \emph{asymptotic test ideals} $\tau(\lambda\cdot\parallel D\parallel)$ for every
$\lambda\in\RR_{\geq 0}$. The following is our main technical result (see Theorem~\ref{thm_main}
below).

\noindent {\bf Theorem A}. 
Let $X$ be a smooth projective variety over an algebraically closed field of positive characteristic
and let $H$ be an ample divisor on $X$, with $\cO_X(H)$ globally generated.
If $D$ and $E$ are divisors on $X$ such that $\cO_X(D)$ has non-negative Iitaka dimension, and $\lambda\in\QQ_{\geq 0}$ is such that $E-\lambda D$ is nef, then the sheaf
$$\tau(\lambda\cdot \parallel D\parallel)\otimes_{\cO_X}\cO_X(K_X+E+dH)$$
is globally generated for every $d\geq\dim(X)+1$.

Here $K_X$ denotes a canonical divisor on $X$. For the corresponding result in characteristic
zero, in which $\tau(\lambda\cdot \parallel D\parallel)$ is replaced by the asymptotic multiplier ideal
$\cJ(\lambda\cdot\parallel D\parallel)$, see \cite[Corollary~11.2.13]{positivity}.
It was Schwede who first noticed in \cite{Schwede} that one can use an argument due to Keeler 
\cite{Keeler} and Hara (unpublished) to obtain global generation statements involving test ideals. 
The idea is to use Castelnuovo-Mumford regularity and the fact that by pushing-forward
via the Frobenius morphism one can reduce the desired vanishings to Serre's asymptotic vanishing.
Our argument follows the one in \cite{Schwede}, with some modifications coming from the fact that we need to consider test ideals of not necessarily locally principal ideals, and we have the extra
nef divisor $E-\lambda D$ to deal with (in order to do this, we use Fujita's vanishing theorem
instead of Serre's asymptotic vanishing).

Once we have the above uniform global generation statement and its corollaries, the basic results describing $B_-(D)$ for a big divisor $D$ follow as in \cite{ELMNP}. Recall that given
a closed point $x\in X$\footnote{In the main body of the paper we will deal with an arbitrary irreducible proper closed subset $Z\subset X$, not just with points.} and a big divisor 
$D$ on $X$, one defines
the \emph{asymptotic order of vanishing} $\ord_x(\parallel D\parallel)$ by
$$\ord_x(\parallel D\parallel):=\inf_{m\geq 1}\frac{\ord_x|mD|}{m}=\lim_{m\to\infty}
\frac{\ord_x|mD|}{m},$$
where $\ord_x|mD|$ is the order of vanishing at $x$ of a general element in $|mD|$. 
The following are the main properties of this function (see Theorem~\ref{function_big_cone} below).

\noindent{\bf Theorem B}.
Let $X$ be a smooth projective variety over an algebraically closed field of positive characteristic,
and $x$ a closed point on $X$.
\begin{enumerate}
\item[i)] For every big divisor $D$, the asymptotic order of vanishing $\ord_x(\parallel D\parallel)$ only depends
on the numerical class of $D$.
\item[ii)] The function $D\to\ord_x(\parallel D\parallel)$ extends as a continuous function to the cone of big divisor 
classes ${\rm Big}(X)_{\RR}$.
\end{enumerate}

In characteristic zero, this was also proved by Nakayama in \cite{Nak}. As a consequence of
Theorems~A and B, we obtain the following description of the non-nef locus
(see Theorem~\ref{char_nonnef} below).

\noindent{\bf Theorem C}. 
Let $X$ be a smooth projective variety over an algebraically closed field of positive characteristic
and $x$ a closed point on $X$. For a big divisor $D$, the following are equivalent:
\begin{enumerate}
\item[i)] $x$ does not lie in the non-nef locus $\BB_-(D)$.
\item[ii)] There is a divisor $G$ on $X$ such that $x$ does not lie in the base locus of
$|mD+G|$ for every $m\geq 1$.
\item[iii)] There is a real number $M$ such that $\ord_x|mD|\leq M$ for every $m$ with $|mD|$ non-empty.
\item[iv)] $\ord_x(\parallel D\parallel)=0$.
\item[v)] For every $m\geq 1$, the ideal $\tau(\parallel mD\parallel)$ does not vanish at $x$. 
\end{enumerate}

At the boundary of the big cone, the situation is more complicated. If $D$ is a pseudo-effective
$\RR$-divisor, then $D+A$ is big for every ample $\RR$-divisor $A$. As in \cite{Nak}, we 
define $\sigma_x(D):=\sup_A\ord_x(\parallel D+A\parallel)$, where $A$ varies over all
ample $\RR$-divisors. This is tautologically a lower
semi-continuous function on the pseudo-effective cone, but it is not continuous in general. 
Following an idea of Hacon, we show that for every $\lambda\in\RR_{\geq 0}$, there is a unique
minimal element in the set of ideals $\tau(\lambda\cdot\parallel D+A\parallel)$, where $A$ is as above. We denote this ideal by $\tau_+(\lambda\cdot \parallel D\parallel)$.
The following theorem gives the description of the non-nef locus for pseudo-effective divisors. 

\noindent{\bf Theorem D}.
Let $X$ be a smooth projective variety over an algebraically closed field of positive characteristic
and $x$ a closed point on $X$. If $D$ is a pseudo-effective $\RR$-divisor on $X$, then the following are equivalent:
\begin{enumerate}
\item[i)] $x$ does not lie in the non-nef locus $\BB_-(D)$.
\item[ii)] $\sigma_x(D)=0$.
\item[iii)] For every $m\geq 1$, the ideal $\tau_+(m\cdot \parallel D\parallel)$ does not vanish at $x$. 
\end{enumerate}

The paper is organized as follows. In \S 2 we review, following \cite{ELMNP}, the definition and
elementary properties of the asymptotic order function and of the non-nef locus. In \S 3 we 
recall the definition of test ideals and of its asymptotic version. We prove here that
given an arbitrary graded sequence of ideals, its asymptotic order of vanishing along a subvariety can be computed from the orders of vanishing of the corresponding sequence
of asymptotic test ideals. Section 4 contains the proof of our key technical result, 
Theorem A. Some applications to asymptotic test ideals and their $F$- jumping numbers  
are given in the following section. In \S 6 we deduce the results stated in Theorems B and C above, while the last section of the paper contains the description of the non-nef locus for 
pseudo-effective $\RR$-divisors.

\subsection*{Acknowledgment}
I am indebted to Karl Schwede for several inspiring conversations on the results of his 
preprint \cite{Schwede} and for  comments on a preliminary version of this note. I would also like to thank Rob Lazarsfeld for many discussions 
over the years on multiplier ideals and asymptotic invariants. Last but not least, I am grateful
to the anonymous referees for their comments.

\section{Non-nef loci and asymptotic orders of vanishing}

In this section we review, following \cite{ELMNP}, the definition of the non-nef locus and of the asymptotic order of vanishing of a divisor along a subvariety. Let $X$ be a smooth
variety\footnote{We assume that all varieties are irreducible and reduced.}
over an algebraically closed field $k$ (in this section we make no restriction on the characteristic). 

Recall that a \emph{graded sequence of ideals} on $X$ consists of a sequence
$\fra_{\bullet}=(\fra_m)_{m\geq 1}$ 
of ideals of $\cO_X$ (all ideals are assumed to be coherent) that satisfies
\begin{equation}\label{eq_graded_sequence}
\fra_{m_1}\cdot\fra_{m_2}\subseteq\fra_{m_1+m_2}
\end{equation}
for all $m_1,m_2\geq 1$. We assume that all our graded sequences are \emph{nonzero}, that is,
$\fra_m\neq 0$ for some $m\geq 1$. 

The most interesting examples of graded sequences arise as follows. Suppose that $X$ is 
complete, and that $D$ is a divisor on $X$ such that $\cO_X(D)$ has non-negative 
Iitaka dimension. Let $\fra_{|mD|}$ be the ideal defining the base locus of $\cO_X(mD)$, that is,
evaluation of sections induces a surjective map
$$H^0(X,\cO_X(mD))\otimes\cO_X\to\fra_{|mD|}\cdot \cO_X(mD).$$
In this case $\fra^D_{\bullet}=(\fra_{|mD|})_{m\geq 1}$ is a graded sequence of ideals
(note that some $\fra_{|mD|}$ is nonzero by the assumption on the Iitaka dimension of $\cO_X(D)$). 

Suppose now that $X$ is not necessarily complete and $Z$ is an irreducible proper closed subset of $X$. 
For a nonzero ideal $\fra$ on $Z$ we denote by $\ord_Z(\fra)$ the order of vanishing of $\fra$
along $Z$; in other words, if $\cO_{X,Z}$ is the local ring of $X$ at the generic point of $Z$,
having maximal ideal $\frm_Z$, then $\ord_Z(\fra)$ is the largest $r$ such that
$\fra\cdot\cO_{X,Z}\subseteq\frm_Z^r$. By convention, we put
$\ord_Z(0)=\infty$. It is clear that given two ideals $\fra$ and $\fra'$
on $X$, we have
\begin{equation}\label{ord_subadditive}
\ord_Z(\fra\cdot\fra')=\ord_Z(\fra)+\ord_Z(\fra').
\end{equation}

If $\fra_{\bullet}$ is a graded sequence of ideals on $X$, then the
\emph{asymptotic order of vanishing} of $\fra_{\bullet}$ along $Z$ is
$$\ord_Z(\fra_{\bullet}):=\inf_{m\geq 1}\frac{\ord_Z(\fra_m)}{m}.$$
It is easy to deduce from properties (\ref{eq_graded_sequence}) and (\ref{ord_subadditive})
that $\ord_Z(\fra_{\bullet})=\lim_{m\to\infty}\frac{\ord_Z(\fra_m)}{m}$, where the limit is over those
$m$ such that $\fra_m$ is nonzero (see for example \cite[Lemma~2.3]{JM}).

If $X$ is complete and $D$ is a divisor on $X$ such that $\cO_X(D)$ has non-negative
Iitaka dimension, then we consider the graded sequence $\fra^D_{\bullet}=
(\fra_{|mD|})_{m\geq 1}$. The asymptotic order of vanishing $\ord_Z(\fra^D_{\bullet})$ is denoted by $\ord_Z(\parallel D\parallel)$. 
Since $\ord_Z(\parallel D\parallel)$ is the limit of the corresponding normalized orders of vanishing, we have the equality $\ord_Z(\parallel mD\parallel)
=m\cdot\ord_Z(\parallel D\parallel)$ for every positive integer $m$.
Given a $\QQ$-divisor $D$ with $h^0(X,\cO_X(mD))\neq 0$
for some positive integer $m$ such that $mD$ has integer coefficients, we can therefore define 
$\ord_Z(\parallel D\parallel):=\frac{1}{m}\ord_Z(\parallel mD\parallel)$. It is clear that
this is well-defined and 
$\ord_Z(\parallel\lambda D\parallel)=\lambda\cdot\ord_Z(\parallel D\parallel)$ 
for every $\lambda\in\QQ_{\geq 0}$.

\begin{remark}
If $D$ and $E$ are divisors on $X$ such that both $\cO_X(D)$ and $\cO_X(E)$
have non-negative Iitaka dimension, we have
$\fra_{|mD|}\cdot\fra_{|mE|}\subseteq \fra_{|m(D+E)|}$
for every $m$.
This easily implies 
$$\ord_Z(\parallel D+E\parallel)\leq \ord_Z(\parallel D\parallel)+\ord_Z(\parallel E\parallel).$$
\end{remark}

We now turn to the definition of the stable base locus and of the non-nef locus.
Suppose that $X$ is a smooth projective variety. We denote by 
$\NS^1(X)_{\RR}$ the finite-dimensional real vector space of numerical equivalence classes
of $\RR$-divisors on $X$, and by
${\rm Big}(X)_{\RR}$ the \emph{big cone}, that is, the
open cone of big $\RR$-divisor classes. The closure of the big cone is the cone of
\emph{pseudo-effective} divisor classes.

For a divisor $D$ on $X$, we denote by 
${\rm Bs}(D)$ the base-locus of $|D|$ (with the reduced scheme structure). It is clear 
that for every positive integers $m$ and $r$, we have ${\rm Bs}(rmD)\subseteq {\rm Bs}(mD)$,
hence the Noetherian property implies that the intersection 
$\bigcap_{m\geq 1}{\rm Bs}(mD)$ is equal to ${\rm Bs}(\ell D)$ if $\ell$ is divisible enough. This is the \emph{stable base locus} of $D$, denoted by $\BB(D)$. 
It is clear that $\BB(D)=\BB(rD)$ for every positive integer $r$. Therefore we can define
$\BB(D)$ for a $\QQ$-divisor $D$ as $\BB(rD)$, where $r$ is any positive integer such that
$rD$ has integer coefficients.

Suppose now that $D$ is an $\RR$-divisor on $X$. The \emph{non-nef locus} of $D$
(called \emph{restricted base locus} in \cite{ELMNP}) is the union
$$\BB_-(D)=\bigcup_A\BB(D+A),$$
where the union is over all ample $\RR$-divisors $A$ such that $D+A$ is a $\QQ$-divisor. 
The properties in the following proposition are simple consequences of the definition,
see \cite[\S 1]{ELMNP}.

\begin{proposition}\label{properties_nonnef}
Let $D_1$ and $D_2$ be $\RR$-divisors on $X$.
\item[i)] $\BB_-(D_1)=\BB_-(\lambda D_1)$ for every $\lambda>0$.
\item[ii)] If $D_1$ and $D_2$ are numerically equivalent, then $\BB_-(D_1)=\BB_-(D_2)$. 
\item[iii)] The non-nef locus $\BB_-(D_1)$ is empty if and only if $D_1$ is nef.
\item[iv)] If $D_1$ is a $\QQ$-divisor, then $\BB_-(D_1)\subseteq\BB(D_1)$. 
\item[v)] We have $\BB_-(D_1+D_2)\subseteq\BB_-(D_1)\cup\BB_-(D_2)$.
\item[vi)] If $(A_m)_{m\geq 1}$ is a sequence of ample $\RR$-divisors with each $D+A_m$
having rational coefficients, and such that the classes of the $A_m$ converge to zero in $\NS^1(X)_{\RR}$,
then $\BB_-(D)=\bigcup_{m\geq 1}\BB(D+A_m)=\bigcup_{m\geq 1}\BB_-(D+A_m)$. 
\end{proposition}

It is not known whether $\BB_-(D)$ is always a Zariski closed subset of $X$, though
property vi) above shows that it is a countable union of closed subsets. This property also implies
that if the ground field $k$  is uncountable, then $\BB_-(D)=X$ if and only if $D$ is not pseudo-effective.

\section{Asymptotic test ideals}

We start by reviewing the definition of test ideals.
These ideals have been introduced and studied in \cite{HY}.
 Since we  only deal with smooth varieties, we use an alternative definition from \cite{BMS},
 which is more suitable for our applications (this description goes back to 
 \cite[Lemma~2.1]{HT}). 
Suppose that
$X$ is a smooth $n$-dimensional variety over an algebraically closed field $k$ of characteristic $p>0$ (in fact, for what follows it is enough to assume $k$ perfect). Let $\omega_X$
denote the sheaf of $n$-forms on $X$.
We denote by $F\colon X\to X$ the Frobenius morphism, that is given by the identity on the topological space, and by taking the $p$-power on regular functions. 

The key object is the \emph{trace} map ${\rm Tr}={\rm Tr}_X\colon F_*(\omega_X)\to\omega_X$. This is a surjective map that can be either defined as a trace map for duality with respect to $F$, or as coming from the Cartier isomorphism. Given algebraic coordinates $x_1,\ldots,x_n$ on an open subset $U$ of $X$, the trace map is characterized by
$${\rm Tr}(x_1^{i_1}\cdots x_n^{i_n}dx_1\wedge\cdots\wedge dx_n)=
x_1^{\frac{i_1-p+1}{p}}\cdots x_n^{\frac{i_n-p+1}{p}}dx_1\wedge\cdots\wedge dx_n,$$
where the monomial on the right-hand side is understood to be zero if one of the exponents is not an integer.
Iterating this map $e$ times we obtain a surjective map ${\rm Tr}^e\colon F^e_*(\omega_X)\to\omega_X$.

Given an ideal $\frb$ in $\cO_X$ and $e\geq 1$, the image ${\rm Tr}^e(\frb\cdot\omega_X)$
can be written as $\frb^{[1/p^e]} \cdot\omega_X$ for some ideal $\frb^{[1/p^e]}$ in $\cO_X$.
This is not the definition in \cite{BMS}, but it can be easily seen to be equivalent to the definition
therein via \cite[Proposition~2.5]{BMS}. For example, when $\fra$ is the ideal defining a smooth divisor $E$ on $X$, we have $(\fra^m)^{[1/p^e]}=\cO_X(-\lfloor m/p^e\rfloor E)$, 
where $\lfloor u\rfloor$
 denotes the largest integer $\leq u$.

Given a nonzero ideal $\fra$ on $X$ and $\lambda\in\RR_{\geq 0}$, one shows that
$$\left(\fra^{\lceil\lambda p^e\rceil}\right)^{[1/p^e]}\subseteq
\left(\fra^{\lceil\lambda p^{e+1}\rceil}\right)^{[1/p^{e+1}]}$$
for every $e\geq 1$. Here we put $\lceil u\rceil$ for the smallest integer $\geq u$.
By the Noetherian property, there is an ideal $\tau(\fra^{\lambda})$, the
\emph{test ideal} of $\fra$ of exponent $\lambda$, that is equal to
 $\left(\fra^{\lceil\lambda p^e\rceil}\right)^{[1/p^e]}$ for $e\gg 0$. 
 One can show that if $r$ is a positive integer, then $\tau(\fra^{r\lambda})=
 \tau((\fra^r)^{\lambda})$. Furthermore, we have $\fra\subseteq\tau(\fra)$.

 Test ideals share many of the properties of the multiplier ideals. 
 If $\fra\subseteq\frb$, then $\tau(\fra^{\lambda})\subseteq\tau(\frb^{\lambda})$ for every 
 $\lambda$. We have 
 $\tau(\fra^{\lambda})=\cO_X$ if $0\leq\lambda\ll 1$. For $\lambda>\mu$
 we have $\tau(\fra^{\lambda})\subseteq\tau(\fra^{\mu})$. Given $\lambda\geq 0$, there is
 $\epsilon>0$ such that $\tau(\fra^{\lambda})=\tau(\fra^{\mu})$ for $\mu$ with
 $\lambda\leq \mu\leq\lambda+\epsilon$. One says that $\lambda>0$ is an 
 $F$-\emph{jumping number} of $\fra$ if $\tau(\fra^{\lambda})\neq\tau(\fra^{\lambda'})$
 for every $\lambda'<\lambda$. It is known that the set of $F$-jumping numbers of
 $\fra$ is a discrete set of rational numbers. For the proof of all these properties, we refer to 
 \cite{BMS}. 
 
 A nice feature of the theory is that
 the Subadditivity Theorem for multiplier ideals (see \cite[Theorem~9.5.20]{positivity})
 has an analogue in this setting. This says that for every nonzero ideals $\fra$ and $\frb$
 and every $\lambda\geq 0$, we have
 \begin{equation}\label{subadditivity}
 \tau((\fra\cdot\frb)^{\lambda})\subseteq\tau(\fra^{\lambda})\cdot\tau(\frb^{\lambda}).
 \end{equation}
 In particular, we have $\tau(\fra^{m\lambda})\subseteq\tau(\fra^{\lambda})^m$
 for every positive integer $m$.
 For a proof see \cite[Proposition~2.11]{BMS}. 
 
 One can similarly define a mixed test ideal: given nonzero ideals $\fra$ and $\frb$
 in $\cO_X$ and $\lambda,\mu\in\RR_{\geq 0}$, there is an ideal $\tau(\fra^{\lambda}\frb^{\mu})$
 that is equal to $(\fra^{\lceil\lambda p^e\rceil}\frb^{\lceil\mu p^e\rceil})^{[1/p^e]}$
 for $e\gg 0$. One can show that for every $\lambda$ and $\mu$, there is $\epsilon>0$ such that
 $\tau(\fra^{\lambda}\frb^{\mu})=\tau(\fra^{\lambda'}\frb^{\mu'})$ if $\lambda\leq\lambda'\leq
 \lambda+\epsilon$ and $\mu\leq\mu'\leq\mu+\epsilon$ (this follows by adapting the argument
 in the non-mixed case, see the proof of \cite[Proposition~2.14]{BMS}).
 The notion of mixed test ideals will only
 come up in the proof of Proposition~\ref{asymptotic_test} iv) below.

It is straightforward to define an asymptotic version of test ideals, proceeding in the same way as in the case of multiplier ideals (this has been noticed already in \cite{Hara}, where the construction was used to compare symbolic powers with usual powers of ideals in positive characteristic). Suppose that
$\fra_{\bullet}$ is a graded sequence of ideals on $X$ (recall that we always assume
that some $\fra_m$ is nonzero) and $\lambda\in\RR_{\geq 0}$. If $m$ and $r$ are positive integers
such that $\fra_m$ is nonzero, then
$$\tau(\fra_m^{\lambda/m})=\tau((\fra_m^r)^{\lambda/mr})\subseteq
\tau(\fra_{mr}^{\lambda/mr}),$$
where the inclusion follows from $\fra_m^r\subseteq\fra_{mr}$. It follows from the Noetherian
property that there is a unique ideal $\tau(\fra_{\bullet}^{\lambda})$ such that
$\tau(\fra_m^{\lambda/m})\subseteq \tau(\fra_{\bullet}^{\lambda})$
for all $m$ (such that $\fra_m$ is nonzero), with equality if $m$ is divisible enough.
This is the \emph{asymptotic test ideal} of $\fra_{\bullet}$ of exponent $\lambda$. 
We collect in the next proposition some easy properties of asymptotic test ideals.
The proof follows the case of multiplier ideals (see \cite[\S 11.2]{positivity}).

\begin{proposition}\label{asymptotic_test}
Let $\fra_{\bullet}$ and $\frb_{\bullet}$ be two graded sequences of ideals on $X$.
\begin{enumerate}
\item[i)] We have $\tau(\fra_{\bullet}^{\lambda})\subseteq\tau(\fra_{\bullet}^{\mu})$ for every
$\lambda\geq \mu$.
\item[ii)] We have $\tau(\fra_{\bullet}^{m\lambda})\subseteq\tau(\fra_{\bullet}^{\lambda})^m$
for all positive integers $m$.
\item[iii)] For every $\lambda\in\RR_{\geq 0}$, there is $\epsilon>0$ such that
$\tau(\fra_{\bullet}^{\lambda})=\tau(\fra_{\bullet}^{\mu})$ for all $\mu$ with 
$\lambda\leq\mu\leq\lambda+\epsilon$.
\item[iv)] If there is a nonzero ideal $\frc$ on $X$ such that $\frc\cdot\fra_m\subseteq\frb_m$ for all $m\gg 0$, then $\tau(\fra_{\bullet}^{\lambda})\subseteq\tau(\frb_{\bullet}^{\lambda})$ for all $\lambda\in\RR_{\geq 0}$. 
\end{enumerate}
\end{proposition}

\begin{proof}
The assertions in i) and ii) follow from the definition of asymptotic test ideals, using 
the corresponding properties in the case of test ideals. For iii), let $m$ be such that
$\tau(\fra_{\bullet}^{\lambda})=\tau(\fra_m^{\lambda/m})$. There is $\epsilon>0$ such that
$\tau(\fra_m^{\lambda/m})=\tau(\fra_m^{(\lambda+\epsilon)/m})\subseteq\tau(\fra_{\bullet}^{\lambda+\epsilon})$, which proves iii).

For iv), given $\lambda\in\RR_{\geq 0}$ let us choose $m$ such that 
$\tau(\fra_{\bullet}^{\lambda})=\tau(\fra_m^{\lambda/m})$. If $\ell\gg 0$, then
$$\tau(\fra_{\bullet}^{\lambda})=\tau(\fra_m^{\lambda/m})=\tau(\frc^{\lambda/m\ell}\fra_m^{\lambda/m})=\tau((\frc\fra_m^{\ell})^{\lambda/m\ell})\subseteq\tau((\frc\fra_{m\ell})^{\lambda/m\ell})
\subseteq\tau(\frb_{m\ell}^{\lambda/m\ell})\subseteq\tau(\frb_{\bullet}^{\lambda}).$$
\end{proof}

We say that $\lambda>0$ is an $F$-jumping number of $\fra_{\bullet}$ if 
$\tau(\fra_{\bullet}^{\lambda})\neq\tau(\fra_{\bullet}^{\mu})$ for every $\mu<\lambda$.
We will see in \S 5 that if $\fra_{\bullet}$ is associated to a divisor on a projective variety, then
the set of $F$-jumping numbers of $\fra_{\bullet}$ is discrete. 

If $\fra_{\bullet}$ is a graded sequence as above, we also consider 
$\frb_{\bullet}=(\frb_m)_{m\geq 1}$, where $\frb_m=\tau(\fra_{\bullet}^m)$.
Note that $\fra_m\subseteq\tau(\fra_m)\subseteq\frb_m$ for every $m$. 
The sequence $\frb_{\bullet}$ is not a graded sequence, but Proposition~\ref{asymptotic_test} ii) implies that
$\frb_{mr}\subseteq\frb_m^r$ for every $m,r\geq 1$. Furthermore, we have
$\frb_{m_1}\subseteq\frb_{m_2}$ for $m_1>m_2$. It is easy to deduce that if $Z$ is an irreducible 
proper closed subset of $X$, then
$$\ord_Z(\frb_{\bullet}):=\sup_m\frac{\ord_Z(\frb_m)}{m}=\lim_{m\to\infty}
\frac{\ord_Z(\frb_m)}{m}$$
(see \cite[Lemma~2.6]{JM}). 

\begin{proposition}\label{compute_test}
If $\fra_{\bullet}$ is a graded sequence of ideals on $X$ and $\frb_{\bullet}$ is the corresponding
sequence of asymptotic test ideals, then for every irreducible proper closed subset $Z$ of $X$,
we have $\ord_Z(\fra_{\bullet})=\ord_Z(\frb_{\bullet})$. 
\end{proposition}

\begin{proof}
Since $\fra_m\subseteq\frb_m$ for every $m$, it is clear that we have 
$\ord_Z(\fra_m)\geq\ord_Z(\frb_m)$ for all $m$, hence $\ord_Z(\fra_{\bullet})\geq
\ord_Z(\frb_{\bullet})$. In order to prove the reverse inequality, given $m\geq 1$, let us choose
$r$ such that $\frb_m=\tau(\fra_{mr}^{1/r})$. It follows from Proposition~\ref{estimate_order}
below that 
$$\ord_Z(\frb_{\bullet})\geq \frac{\ord_Z(\frb_m)}{m}>\frac{\ord_Z(\fra_{mr})}{mr}-
\frac{{\rm codim}(Z,X)}{m}\geq\ord_Z(\fra_{\bullet})-\frac{{\rm codim(Z,X})}{m}.$$
Since this holds for every $m\geq 1$, we deduce that $\ord_Z(\frb_{\bullet})\geq
\ord_Z(\fra_{\bullet})$, which completes the proof of the proposition.
\end{proof}

The next proposition is an instance of the fact that ``the test ideal is contained in the
multiplier ideal", which goes back to \cite[Theorem~3.4]{HY}. We give a direct proof, since
the argument is particularly transparent in our setting.

\begin{proposition}\label{estimate_order}
If $\fra$ is a nonzero ideal on $X$ and $Z$ is an irreducible proper closed subset of $X$,
then for every $\lambda\in\RR_{\geq 0}$ we have
$$\ord_Z(\tau(\fra^{\lambda}))>\lambda\cdot\ord_Z(\fra)-{\rm codim}(Z,X).$$
\end{proposition}

\begin{proof}
Since construction of test ideals commutes with restriction to an open subset, after replacing
$X$ by a suitable open neighborhood of the generic point of $Z$, we may assume that $Z$
is smooth. Let $\pi\colon Y\to X$ be the blow-up of $X$ along $Z$, with exceptional divisor $E$. 
If $c={\rm codim}(Z,X)$, then the relative canonical divisor $K_{Y/X}$ is equal to $(c-1)E$. 

We have a commutative diagram
\begin{equation}\label{diag4_2}
\begin{CD}
F^e_*\pi_*(\omega_Y) @>{\pi_*({\rm Tr}_Y^e)}>>
\pi_*(\omega_Y)\\
@V{F^e_*(\rho)}VV @VV{\rho}V\\
F^e_*(\omega_X)@>{{\rm Tr}_X^e}>>\omega_X
\end{CD}
\end{equation}
in which the vertical maps are isomorphisms. Note that if $J$ is an ideal in $\cO_Y$, then
$\rho(\pi_*(J\cdot\omega_Y))=\pi_*(J\cdot\cO_Y(K_{Y/X}))\cdot\omega_X$. 

Given the nonzero ideal $\fra$ in $\cO_X$, we put $\frb=\fra\cdot\cO_Y$, and consider
$M:=F^e_*(\fra^m\cdot\omega_X)$. Since $K_{Y/X}$ is effective, we have
$F^e_*(\rho)^{-1}(M)\subseteq F^e_*\pi_*(\frb^m\cdot \omega_Y)$, and 
using 
the commutativity of the above diagram to compute ${\rm Tr}_X^e(M)$
gives
$$(\fra^m)^{[1/p^e]}\subseteq\pi_*(\cO_Y(K_{Y/X})\cdot (\frb^m)^{[1/p^e]}).$$
If $s=\ord_Z(\fra)$, then $\frb\subseteq\cO_Y(-sE)$, and since $E$ is nonsingular we have
$$\cO_Y(K_{Y/X})\cdot (\frb^m)^{[1/p^e]}\subseteq\cO_Y((c-1-\lfloor ms/p^e\rfloor) E).$$
For a fixed $\lambda\in\RR_{\geq 0}$, let us take $m=\lceil \lambda p^e\rceil$ with $e\gg 0$, 
so that $\lfloor ms/p^e\rfloor=\lfloor\lambda s\rfloor$, and we conclude
$$\ord_Z(\tau(\fra^{\lambda}))=\ord_Z((\fra^{\lceil\lambda p^e\rceil})^{[1/p^e]})\geq \lfloor \lambda s\rfloor
-c+1,$$
which is equivalent with the inequality in the proposition.
\end{proof}

In the following sections we will be interested in the case when
$X$ is projective and $D$ is a divisor on $X$ such that $h^0(X,\cO_X(mD))\neq 0$ 
for some positive integer $m$. We then
denote by $\tau(\lambda\cdot\parallel D\parallel)$ the asymptotic 
test ideal of exponent $\lambda$ associated to the graded sequence $(\fra_{|mD|})_{m\geq 1}$. 
If $D$ is a $\QQ$-divisor such that $h^0(X,\cO_X(mD))\neq 0$ for some $m$ such that $mD$
is integral, then we put $\tau(\lambda\cdot\parallel D\parallel):=
\tau(\lambda/r\cdot \parallel rD\parallel)$ for every $r$ such that $rD$ has integer coefficients.
If $\lambda=1$, then we simply write $\tau(\parallel D\parallel)$.
It is clear from definition that if $\lambda\in\QQ_{\geq 0}$, then 
$\tau(\lambda\cdot \parallel D\parallel)=\tau(\parallel \lambda D\parallel)$
(note that when $\lambda=0$, both sides are trivially equal to $\cO_X$). 
In particular, we see using Proposition~\ref{asymptotic_test} i) that if $D$ is as above 
and $\lambda_1\leq\lambda_2$ are in $\QQ_{\geq 0}$, then $\tau(\parallel \lambda_2 D\parallel)
\subseteq\tau(\parallel \lambda_1 D\parallel)$.

\section{A uniform global generation result}

In this section we prove the main technical result of the paper.  Let $X$ be a smooth projective variety over an algebraically closed field $k$ of positive characteristic. We denote by 
$K_X$ a canonical divisor (that is, we have $\cO_X(K_X)\simeq\omega_X$).
We put $n=\dim(X)$, and consider an ample divisor $H$ on $X$, such that $\cO_X(H)$ is
globally generated. 

\begin{theorem}\label{thm_main}
With the above notation, let $D$ and $E$ be divisors on $X$, and $\lambda\in\QQ_{\geq 0}$.
If $\cO_X(D)$ has non-negative Iitaka dimension, and $E-\lambda D$ is nef, then
the sheaf
$$\tau(\lambda\cdot\parallel D\parallel)\otimes_{\cO_X} \cO_X(K_X+E+dH)$$
is globally generated for every $d\geq n+1$.
\end{theorem}

The proof that we give below follows  the proof of 
\cite[Theorem~4.3]{Schwede}, which in turn makes use of an argument of
 Keeler \cite{Keeler} and Hara (unpublished).
In our proof we also use of the following theorem of Fujita \cite{Fujita}.
If $\cF$ is a coherent sheaf on $X$ and $A$ is an ample divisor, then there is $\ell_0$
such that $H^i(X,\cF\otimes_{\cO_X}\cO_X(\ell A+P))=0$ for every $i\geq 1$, every $\ell\geq\ell_0$, and every
nef divisor $P$. 

\begin{proof}[Proof of Theorem~\ref{thm_main}]
By definition of the asymptotic test ideal, we can find $m$ such that if 
$\fra_m$ is the ideal defining the base-locus of $|mD|$,  then $\tau(\lambda\cdot\parallel D\parallel)=
\tau(\fra_m^{\lambda/m})$. Let us fix such $m$. 
For $e\gg 0$ we have $\tau(\fra_m^{\lambda/m})=(\fra_m^{\lceil \lambda p^e/m\rceil})^{[1/p^e]}$, and therefore
there is a surjective map
\begin{equation}\label{thm_main_eq1}
F^e_*(\fra_m^{\lceil \lambda p^e/m\rceil}\otimes_{\cO_X}\cO_X(K_X))\to
\tau(\lambda\cdot\parallel D\parallel)\otimes_{\cO_X}\cO_X(K_X).
\end{equation}
By tensoring with $\cO_X(E+dH)$ and using the projection formula, we obtain the surjective map
\begin{equation}\label{thm_main_eq2}
F^e_*(\fra_m^{\lceil \lambda p^e/m\rceil}\otimes_{\cO_X}\cO_X(K_X+p^e(E+dH)))
\to\tau(\lambda\cdot\parallel D\parallel)\otimes_{\cO_X}\cO_X(K_X+E+dH).
\end{equation}
On the other hand, we have by definition a surjective map 
$H^0(X,\cO_X(mD))\otimes_k\cO_X(-mD)\to\fra_m$,
hence a surjective map 
\begin{equation}\label{eq10_thm_main}
W\otimes_k\cO_X(-m\lceil\lambda p^e/m\rceil D)\to
\fra_m^{\lceil\lambda p^e/m\rceil},
\end{equation}
 where $W=
{\rm Sym}^{\lceil \lambda p^e/m\rceil}H^0(X,\cO_X(mD))$.
Tensoring (\ref{eq10_thm_main}) with $\cO_X(K_X+p^e(E+dH))$ and
pushing forward by $F^e$ (note that $F^e_*$ is exact since the Frobenius morphism is affine),
we obtain a surjective map
\begin{equation}\label{thm_main_eq3}
W\otimes_k F^e_*\cO_X(K_X+p^e(E+dH)-m\lceil\lambda p^e/m\rceil D)\to
F^e_*(\fra_m^{\lceil \lambda p^e/m\rceil}\otimes_{\cO_X}\cO_X(K_X+p^e(E+dH))).
\end{equation}
It follows from the surjective maps (\ref{thm_main_eq2}) and (\ref{thm_main_eq3})
that in order to complete the proof of the theorem, it is enough to show that for $e\gg 0$,
the sheaf 
$$F^e_*\cO_X(K_X+p^e(E+dH)-m\lceil\lambda p^e/m\rceil D)$$
is globally generated. In fact, it is enough to show that this sheaf is $0$-regular in the sense
of Castelnuovo-Mumford regularity with respect to the ample globally generated line bundle
$\cO_X(H)$ (we refer to \cite[\S 1.8]{positivity} for basic facts about Castelnuovo-Mumford regularity).
Therefore, it is enough to show that if $e\gg 0$, then
\begin{equation}\label{thm_main_eq4}
H^i(X, \cO_X(-iH)\otimes_{\cO_X}F^e_*\cO_X(K_X+p^e(E+dH)-m\lceil\lambda p^e/m\rceil D))=0
\end{equation}
for all $i$ with $1\leq i\leq n$. 

Using again the projection formula and the fact that $F^e$ is affine, we obtain
$$H^i(X, \cO_X(-iH)\otimes_{\cO_X}F^e_*\cO_X(K_X+p^e(E+dH)-
m\lceil\lambda p^e/m\rceil D))$$
$$\simeq H^i(X,F^e_*\cO_X(K_X+p^e(E+(d-i)H)-m\lceil\lambda p^e/m\rceil D))$$
$$\simeq H^i(X,\cO_X(K_X+p^e(E+(d-i)H)-m\lceil\lambda p^e/m\rceil D)).$$
Note that by assumption $d-i\geq 1$ for $i\leq n$. 

\noindent{\bf Claim}. We can find finitely many divisors $T_1,\ldots,T_r$ on $X$ that satisfy
the following property:
for every $e$, there is $j$ such that the difference $p^eE-m\lceil \lambda p^e/m\rceil D-T_j$ is
 nef.
If this is the case, then by applying Fujita's vanishing theorem to each of the sheaves
$\cF_j=\cO_X(K_X+T_j)$ and to the ample divisor $(d-i)H$ we obtain
$$H^i(X,\cO_X(K_X+p^e(E+(d-i)H)-m\lceil\lambda p^e/m\rceil D))=0$$
for all $i$ with $1\leq i\leq n$ and all $e\gg 0$. Therefore in order to finish the proof, it is enough to check the assertion in the claim.

Let us write $\lambda=\frac{a}{b}$, for non-negative integers $a$ and $b$, with $b$ nonzero. 
For every $e\geq 1$, we write $p^e=mbs+t$, for non-negative integers $s$ and $t$, with
$t<mb$. In this case $\lceil \lambda p^e/m\rceil=as+\lceil \frac{at}{bm}\rceil$, hence
$$p^eE-m\lceil\lambda p^e/m\rceil D=ms(bE-aD)+\left(tE-m\left\lceil\frac{at}{bm}\right\rceil 
D\right),$$
and the claim follows since $bE-aD$ is nef by assumption, and $t$ can only take finitely many values. This completes the proof of the theorem.
\end{proof}

\begin{remark}
In Theorem~\ref{thm_main}, we may allow $D$ to be a $\QQ$-divisor:
in this case  we may simply replace $D$ by $mD$
and $\lambda$ by $\lambda/m$, with $m$ divisible enough.
\end{remark}

\section{Applications to asymptotic test ideals of divisors}

In this section we give some consequences of Theorem~\ref{thm_main} to general properties
of asymptotic test ideals. From now on, we always assume that $X$ is a smooth projective variety
over an algebraically closed field $k$ of characteristic $p>0$. Our first result says that the asymptotic test ideals of a big $\QQ$-divisor only depend on the numerical equivalence class
of the divisor.

\begin{proposition}\label{num_equivalence}
If $D$ and $E$ are numerically equivalent big $\QQ$-divisors on $X$, then 
$$\tau(\lambda\cdot\parallel D\parallel)=\tau(\lambda\cdot\parallel E\parallel)$$ for every
$\lambda\in\RR_{\geq 0}$.
\end{proposition}

\begin{proof}
The proof follows as in the case of multiplier ideals in characteristic zero,
see \cite[Example~11.3.12]{positivity}. After replacing $D$ and $E$ by multiples, we may clearly assume that both $D$ and $E$ have integer coefficients. 
Let $H$ be a very ample divisor and $n=\dim(X)$. 
Since $D$ is big, there is a positive integer $\ell$ such that
$\ell D-(K_X+(n+1)H)$ is linearly equivalent with an effective divisor $G$. It follows from
Theorem~\ref{thm_main} that $\tau(\parallel (m-\ell)D\parallel)\otimes_{\cO_X}
\cO_X(mE-G)$ is globally generated for every $m\geq\ell$, hence 
$\tau(\parallel(m-\ell)D\parallel)$
is contained in the ideal $\fra_{|mE-G|}$. Therefore
$$\cO_X(-G)\cdot\tau(\parallel (m-\ell)D\parallel)
\subseteq\fra_{|G|}\cdot\fra_{|mE-G|}
\subseteq\fra_{|mE|}$$
and we deduce
$$\cO_X(-G)\cdot\fra_{|mD|}\subseteq \cO_X(-G)\cdot \tau(\parallel mD\parallel)
\subseteq  \cO_X(-G)\cdot \tau(\parallel (m-\ell)D\parallel)\subseteq \fra_{|mE|}$$
for every $m\geq \ell$. 
Proposition~\ref{asymptotic_test} iv) implies 
$\tau(\lambda\cdot\parallel D\parallel)\subseteq\tau(\lambda
\cdot\parallel E\parallel)$ for every $\lambda\in\RR_{\geq 0}$, and the reverse inclusion
follows by symmetry.
\end{proof}

\begin{proposition}\label{jumps}
If $D$ is a divisor on $X$ such that $\cO_X(D)$ has non-negative Iitaka dimension, then
the set of $F$-jumping numbers of the graded sequence of ideals $(\fra_{|mD|})_{m\geq 1}$
is discrete.
\end{proposition}

\begin{proof}
It is enough to show that for every $\lambda_0>0$, there are only finitely many different values
for $\tau(\lambda\cdot\parallel D\parallel)$, with $\lambda<\lambda_0$. Furthermore, it 
follows from Proposition~\ref{asymptotic_test} iii) that it is
 enough to only consider $\lambda\in\QQ_{\geq 0}$.

Let $H$ be a very ample divisor on $X$ and let $\dim(X)=n$. We also fix a divisor $A$ such that
both $A$ and $A-\lambda_0D$ are nef (for example, $A$ could be a large multiple of an ample divisor).
In this case $A-\lambda D$ is nef for every $\lambda$ with $0\leq \lambda\leq\lambda_0$. 
For every such $\lambda$ which is rational, Theorem~\ref{thm_main} implies that
$$\tau(\lambda\cdot \parallel D\parallel)\otimes_{\cO_X}\cO_X(K_X+A+(n+1)H)$$
is globally generated. In particular, its space of global sections $V_{\lambda}$, which is a linear subspace of
$V:=H^0(X,\cO_X(K_X+A+(n+1)H))$, determines $\tau(\lambda \cdot \parallel D\parallel)$.
Furthermore, if $\lambda_r<\ldots<\lambda_1<\lambda_0$, then 
$V_{\lambda_1}\subseteq\ldots\subseteq V_{\lambda_r}$. Since $V$ is finite-dimensional, this clearly bounds the number of distinct values for $\tau(\lambda\cdot \parallel D\parallel)$ with
$\lambda<\lambda_0$.
\end{proof}

In characteristic zero, Hacon used global generation results to attach a type of 
asymptotic multiplier ideal to a pseudo-effective divisor. The analogous construction
works also in positive characteristic, as follows. Suppose that $D$ is a pseudo-effective $\RR$-divisor
on $X$. 
For every ample $\RR$-divisor $A$, the sum $D+A$ is big. In particular, if $D+A$ is a $\QQ$-divisor, then we may consider
$\tau(\lambda\cdot \parallel D+A\parallel)$ for every $\lambda\in\RR_{\geq 0}$.

\begin{proposition}\label{Hacon_test_ideal}
For every pseudo-effective $\RR$-divisor $D$ and all $\lambda\in\RR_{\geq 0}$, there is a unique minimal element, that we denote by
$\tau_+(\lambda\cdot\parallel D\parallel)$, among all ideals of the form 
$\tau(\lambda\cdot \parallel D+A\parallel)$, where $A$ varies over the ample $\RR$-divisors
such that $D+A$ is a $\QQ$-divisor. Furthermore, there is an open neighborhood 
${\mathcal U}$
of the origin in $\NS^1(X)_{\RR}$ such that
$$\tau_+(\lambda\cdot\parallel D\parallel)=\tau(\lambda\cdot \parallel D+A\parallel)$$
for every ample divisor $A$ with $D+A$ a $\QQ$-divisor and such that the class of $A$ lies in 
${\mathcal U}$.
\end{proposition}

\begin{proof}
Note first that if $A_1$ and $A_2$ are ample divisors with both $D+A_1$ and
$D+A_2$ having $\QQ$-coefficients and such that $A_1-A_2$ is ample, then
$\fra_{|m(D+A_2)|}\subseteq \fra_{|m(D+A_1)|}$ for all $m\gg 0$. This implies
$$\tau(\lambda\cdot\parallel D+A_2\parallel)\subseteq
\tau(\lambda\cdot\parallel D+A_1\parallel).$$

 Choose a very ample divisor $H$
on $X$ and put $n=\dim(X)$.
Suppose now that $B$ is a fixed ample divisor such that $B-\lambda D$ is ample.
If $A$ is an ample $\RR$-divisor such that $D+A$ is a $\QQ$-divisor
and $B-\lambda(D+A)$ is ample, then Theorem~\ref{thm_main} implies that 
$$\tau(\lambda\cdot\parallel D+A\parallel)\otimes_{\cO_X}\cO_X(K_X+(n+1)H+B)$$
is globally generated (if $\lambda$ is not rational, then we apply the theorem 
to some rational $\lambda'>\lambda$ such that $B-\lambda'(D+A)$ is still ample
and $\tau(\lambda\cdot\parallel D+A\parallel)=\tau(\lambda'\cdot\parallel D+A\parallel)$).
In particular, we see that $\tau(\lambda\cdot\parallel D+A\parallel)$ is determined by the subspace 
$$W_A:=H^0(X, \tau(\lambda\cdot\parallel D+A\parallel)\otimes_{\cO_X}\cO_X(K_X+(n+1)H+B))
$$
$$\subseteq W:=H^0(X, \cO_X(K_X+(n+1)H+B)).$$

Since $W$ is finite dimensional, we can find some $A$ as above such that $W_A$
is minimal among all such subspaces. Given any ample $A_1$ such that $D+A_1$
is a $\QQ$-divisor, we may choose an ample $A_2$ such that both $A-A_2$ and
$A_1-A_2$ are ample $\QQ$-divisors. As we have seen, this implies
\begin{equation}\label{eq_Hacon_test_ideal}
\tau(\lambda\cdot\parallel D+A_2\parallel)\subseteq\tau(\lambda\cdot\parallel D+A_1\parallel)\,\,
\text{and}\,\,\tau(\lambda\cdot\parallel D+A_2\parallel)\subseteq\tau(\lambda\cdot\parallel D+A\parallel).
\end{equation}
Note that $B-\lambda(D+A_2)$ is ample, and the second inclusion in (\ref{eq_Hacon_test_ideal})
implies in particular that $W_{A_2}\subseteq W_A$. By the minimality in the choice of $A$
we have, in fact, $W_A=W_{A_2}$, and therefore 
$$\tau(\lambda\cdot\parallel D+A_2\parallel)=\tau(\lambda\cdot\parallel D+A\parallel)
\subseteq \tau(\lambda\cdot\parallel D+A_1\parallel).$$
This shows that $\tau(\lambda\cdot\parallel D+A\parallel)$ satisfies the minimality requirement in the proposition.

Suppose now that
${\mathcal U}$
consists of the classes of those $E$ such that 
 $A-E$ is ample. In this case ${\mathcal U}$ is an open neighborhood of the origin in
 $\NS^1(X)_{\RR}$ which 
satisfies the last assertion in the proposition. Indeed, if $A'$ is ample such that
$D+A'$ is a $\QQ$-divisor and the class of $A'$ lies in ${\mathcal U}$, then 
the argument at the beginning of the proof gives the inclusion $\tau(\lambda\cdot\parallel D+A'\parallel)
\subseteq\tau(\lambda\cdot\parallel D+A\parallel)$, while the reverse inclusion follows from the minimality of $\tau(\lambda \cdot\parallel D+A\parallel)$, which we have proved.
\end{proof}

In the next proposition we list several properties of this new version of asymptotic test ideals.

\begin{proposition}\label{properties_new_definition}
Let $D$ be a pseudo-effective $\RR$-divisor on $X$ and let $\lambda\in\RR_{\geq 0}$.
\begin{enumerate}
\item[i)] If $E$ is a pseudo-effective $\RR$-divisor on $X$, numerically equivalent to $D$, then 
$$\tau_+(\lambda\cdot\parallel D\parallel)=\tau_+(\lambda\cdot\parallel E\parallel).$$
\item[ii)] If $\mu\geq\lambda$, then 
$$\tau_+(\mu\cdot\parallel D\parallel)\subseteq
\tau_+(\lambda\cdot\parallel D\parallel).$$
\item[iii)] If $B$ is a nef $\RR$-divisor, then
$$\tau_+(\lambda\cdot \parallel D\parallel)\subseteq 
\tau_+(\lambda\cdot \parallel D+B\parallel).$$
\item[iv)] We have $\tau_+(\lambda\cdot\parallel D\parallel)=
\tau_+(\parallel\lambda D\parallel)$.
\end{enumerate}
\end{proposition}

\begin{proof}
The assertion in i) follows from definition, once we note that if $A$ is ample, then we can write 
$D+A=E+(A+D-E)$ and $A+D-E$ is ample. The inclusion in ii) follows from definition and the 
fact that for every ample $\RR$-divisor $A$ such that $D+A$ has rational coefficients, we have
$$\tau(\mu\cdot\parallel D+A\parallel)\subseteq
\tau(\lambda\cdot\parallel D+A\parallel).$$

In order to prove iii), let $A$ be ample such that $D+B+A$ is a $\QQ$-divisor, and 
the class of $A$
in $\NS^1(X)_{\RR}$ is in a small enough neighborhood of the origin, so that
$$\tau_+(\lambda\cdot \parallel D+B\parallel)=\tau(\lambda\cdot \parallel D+B+A\parallel).$$
Since $B$ is nef, $A+B$ is ample, hence we can find an ample divisor $A'$ such that
$A+B-A'$ is ample and $D+A'$ is a $\QQ$-divisor. In this case
$$\tau(\lambda\cdot \parallel D+B+A\parallel)\supseteq \tau(\lambda\cdot
\parallel D+A'\parallel)\supseteq\tau_+(\lambda\cdot \parallel D\parallel),$$
where the first inclusion follows from the fact that $\fra_{|m(D+A+B)|}\supseteq
\fra_{|m(D+A')|}$ for all $m$ divisible enough. 

Let us now prove iv). 
Suppose that $A$ is an ample divisor such that $D+A$ has rational coefficients and
the class of $A$ in $\NS^1(X)_{\RR}$ lies in a sufficiently small neighborhood of the origin.
If $\lambda'>\lambda$ is rational and close enough to $\lambda$ (depending on $A$), then
$$\tau_+(\lambda\cdot\parallel D\parallel)=\tau(\lambda\cdot\parallel D+A\parallel)=
\tau(\lambda'\cdot \parallel D+A\parallel)=\tau(\parallel \lambda'(D+A)\parallel).$$
On the other hand, the difference $\lambda'(D+A)-\lambda D=(\lambda'-\lambda)D+\lambda'A$
is ample if $\lambda'-\lambda$ is small enough, hence
$$\tau_+(\parallel \lambda D\parallel)\subseteq \tau(\parallel\lambda'(D+A)\parallel)=
\tau_+(\lambda\cdot\parallel D\parallel).$$
In order to prove the reverse inclusion, let us choose an ample $\RR$-divisor $B$ such that
$\lambda D+B$ is a $\QQ$-divisor and $\tau_+(\parallel \lambda D\parallel)=
\tau(\parallel \lambda D+B\parallel)$. Since $B$ is ample, we can choose an ample
$\RR$-divisor $A'$ such that $B-\lambda A'$ is ample and $D+A'$ is a $\QQ$-divisor. 
We can choose now $\mu>\lambda$ such that $\mu\in\QQ$ and $\mu-\lambda$ is small enough,
so that 
$$(\lambda D+B)-\mu(D+A')=(\lambda-\mu)D+(B-\mu A')$$ is ample. Furthermore, since $\mu-\lambda\ll 1$, we have
$$\tau(\lambda\cdot\parallel D+A'\parallel)=\tau(\mu\cdot\parallel D+A'\parallel)=
\tau(\parallel \mu(D+A')\parallel)\subseteq \tau(\parallel \lambda D+B\parallel)=
\tau_+(\parallel \lambda D\parallel),$$
hence by definition we obtain $\tau_+(\lambda\cdot \parallel D\parallel)\subseteq\tau_+(\parallel
\lambda D\parallel)$.
This completes the proof of iv). 
\end{proof}

\begin{remark}
In general, even for a 
big $\QQ$-divisor $D$, the two ideals $\tau_+(\lambda\cdot\parallel D\parallel)$
and
$\tau(\lambda\cdot\parallel D\parallel)$ might be different. Suppose, for example,
 that $\pi\colon X\to W$ is the blow-up
of a smooth projective variety $W$ of dimension $\geq 2$ at a point, and $E$ is the exceptional divisor. 
Let $H$ be a very ample divisor on $W$ such that $\pi^*(H)-E$ is ample. 
Note that for every non-negative integers $r$ and $s$, the ideal
$\fra_{|r\pi^*(H)+sE|}$ is equal to $\cO_X(-sE)$. Using this, 
it is easy to see that if $D=\pi^*(H)+E$, then for every positive integer $m$ we have
$$\tau(m\cdot\parallel D\parallel)=\cO_X(-mE)\,\,\text{and}\,\,
\tau_+(m\cdot\parallel D\parallel)=\cO_X(-(m-1)E).$$
\end{remark}

\section{The non-nef locus of big divisors}

In this section we prove  Theorems~B and C stated in the Introduction (in a more general version, in which we sometimes do not need to restrict to closed points). As in the previous sections,
we assume that $X$ is a smooth projective variety over an algebraically closed field $k$
of characteristic $p>0$. Let $n=\dim(X)$.

\begin{theorem}\label{function_big_cone}
Let $Z$ be an irreducible proper closed subset of $X$.
For a big $\QQ$-divisor $D$, the value $\ord_Z(\parallel D\parallel)$ only depends on the numerical equivalence class of $D$. Furthermore, the function 
$D\to \ord_Z(\parallel D\parallel)$ extends as a continuous function
on ${\rm Big}(X)_{\RR}$, also denoted by 
$\ord_Z(\parallel-\parallel)$.
\end{theorem}

\begin{proof}
For the first assertion, by homogeneity we may assume that $D$ has integer coefficients.
If $\frb_m=\tau(m\cdot\parallel D\parallel)$, then Proposition~\ref{compute_test} implies
$\ord_Z(\parallel D\parallel)=\sup_{m\geq 1}\frac{\ord_Z(\frb_m)}{m}$.
Since the ideals $\frb_m$ only depend on the numerical equivalence class of $D$ by
Proposition~\ref{num_equivalence}, we obtain the first assertion in the theorem.
The second assertion now follows as in \cite[\S 3]{ELMNP}, where the argument is
characteristic-free.
\end{proof}

\begin{theorem}\label{char_nonnef}
Let $Z$ be an irreducible proper closed subset of $X$ and assume that either the ground field
$k$ is uncountable, or $Z$ consists of a point.
If $D$ is a big divisor on $X$, then the following are equivalent:
\begin{enumerate}
\item[i)] $Z$ is not contained in $\BB_-(D)$.
\item[ii)] There is a divisor $G$ on $X$ such that $Z$ is not contained in the base locus of
$|mD+G|$ for every $m\geq 1$. 
\item[iii)] There is a real number $M$ such that $\ord_Z(\fra_{|mD|})\leq M$ for all $m$ with $|mD|$ non-empty.
\item[iv)] $\ord_Z(\parallel D\parallel)=0$.
\item[v)] For every $m\geq 1$, the ideal $\tau(\parallel mD\parallel)$ does not vanish along $Z$.
\end{enumerate}
\end{theorem}

\begin{proof}
The proof is similar to the one in characteristic zero (see \cite[\S 2]{ELMNP}).
With the notation in the proof of Theorem~\ref{function_big_cone}, we see that
$\ord_Z(\parallel D\parallel)=0$ if and only if $\ord_Z(\frb_m)=0$ for every $m\geq 1$. This proves
the equivalence iv)$\Leftrightarrow$v). 

We now prove the implications v)$\Rightarrow$ii)$\Rightarrow$i)$\Rightarrow$iv).
Let us show v)$\Rightarrow$ii). Let $H$ be a very ample divisor on $X$ and let
$G=K_X+(n+1)H$.
It follows from Theorem~\ref{thm_main} that the sheaf
$\tau(\parallel mD\parallel)\otimes_{\cO_X}\cO_X(mD+G)$ is globally generated.
If v) holds, this implies that for every $m\geq 1$ there is a divisor in $|mD+G|$ that does not vanish at the generic point of $Z$, hence ii). 

For the implication ii)$\Rightarrow$i) we make use of the hypothesis on $Z$ and $k$. This implies that if $Z$ is contained in a countable union of closed subsets, then it is contained in one of these sets. Therefore if $Z\subseteq\BB_-(D)$, then there is an ample 
$\QQ$-divisor $A$ such that $Z\subseteq\BB(D+A)$. However, if $G$ is as in ii), then for $m\gg 0$ the divisor $mA-G$ is ample, hence $Z\subseteq \BB(D+A)=\BB(mD+mA)\subseteq\BB(mD+G)$, a contradiction.

We now show
i)$\Rightarrow$iv). Let $(A_i)_{i\geq 1}$ be a sequence of ample $\QQ$-divisors whose classes in $\NS^1(X)_{\RR}$ converge
to zero. It follows from i) that $Z\not\subseteq\BB(D+A_i)$ for any $i$, which in turn implies that $\ord_Z(\fra_{|r(D+A_i)|})=0$ for $r$ divisible enough. We thus deduce that
$\ord_Z(\parallel D+A_i\parallel)=0$ for all $i$, and by continuity of
$\ord_Z(\parallel-\parallel)$ we conclude that $\ord_Z(\parallel D\parallel)=0$. 

In order to complete the proof of the theorem it is enough to also show the
implications ii)$\Rightarrow$iii)$\Rightarrow$iv). Suppose first that ii) holds.
Since $D$ is big, there is a positive integer $m_0$ and an effective divisor $T$ linearly
equivalent to $m_0D-G$. In this case $T+mD+G$ is linearly equivalent to $(m+m_0)D$,
and the assumption in ii) implies $\ord_Z(\fra_{|(m+m_0)D|})\leq\ord_Z\cO_X(-T)$.
This gives 
iii), by taking $M$ to be the maximum of $\ord_Z\cO_X(-T)$ and of those
$\ord_Z(\fra_{|mD|})$, with $m\leq m_0$ and  $|mD|$ non-empty.
 Since iii) clearly implies iv), this completes the proof of the theorem.
\end{proof}

\begin{remark}\label{Q-divisor}
If in Theorem~\ref{char_nonnef} we allow $D$ to have rational coefficients, the equivalence
between i), iv), and v) still holds. Indeed, it is enough to apply the theorem to $rD$, where
$r$ is a positive integer such that $rD$ has integer coefficients.
\end{remark}

\begin{corollary}
If $D$ is a big $\QQ$-divisor on $X$, then $D$ is nef if and only if 
$\tau(\parallel mD\parallel)=\cO_X$ for every $m\geq 1$.
\end{corollary}

\begin{proof}
Note that $D$ is nef if and only if $x\not\in\BB_-(D)$ for every $x\in X$. By Theorem~\ref{char_nonnef}, this is equivalent with the fact that $\tau(\parallel mD\parallel)$ does not vanish at $x$, for every $x\in X$ and every $m\geq 1$.
\end{proof}

\begin{corollary}
If $Z$ and $k$ satisfy the condition in Theorem~\ref{char_nonnef}, then for every big $\RR$-divisor $D$ on $X$, we have $Z\not\subseteq\BB_-(D)$ if and only if 
$\ord_Z(\parallel D\parallel)=0$. 
\end{corollary}

\begin{proof}
If $Z$ is not contained in $\BB_-(D)$, then we obtain $\ord_Z(\parallel D\parallel)=0$
arguing as in the proof of the implication i)$\Rightarrow$iv) in Theorem~\ref{char_nonnef}.
Conversely, suppose that we have $\ord_Z(\parallel D\parallel)=0$. Let us consider a sequence  of ample $\RR$-divisors $(A_m)_{m\geq 1}$ whose classes
 in $\NS^1(X)_{\RR}$ converge to zero and such that all $D+A_m$ are $\QQ$-divisors.
It is easy to see that $\ord_Z(\parallel D+A_m\parallel)\leq\ord_Z(\parallel D\parallel)=0$,
hence applying Theorem~\ref{char_nonnef} (see also Remark~\ref{Q-divisor})
we get $Z\not\subseteq\BB_-(D+A_m)$ for every $m$. 
Under our assumptions on $Z$ and $k$ this implies that $Z$ is not contained in $\BB_-(D)$
(see Proposition~\ref{properties_nonnef} vi)).
\end{proof}

\section{The case of pseudo-effective divisors}

The picture at the boundary of the pseudo-effective cone is more complicated. 
In particular,
the function $\ord_Z(\parallel-\parallel)$ might not admit a continuous extension
to the pseudo-effective cone, see \cite[IV.2.8]{Nak}. 
In this section we explain, following the approach in \cite{Nak}, how the results 
in the previous section need to be modified in this context.

If $D$ is a preudo-effective $\RR$-divisor on $X$, then for every ample $\RR$-divisor $A$, we know that
$D+A$ is big. If $Z$ is an irreducible proper closed subset of $X$, then we put
$$\sigma_Z(D):=\sup_A\ord_Z(\parallel D+A\parallel)\in\RR_{\geq 0}\cup\{\infty\},$$
where the supremum is over all ample $\RR$-divisors $A$. Note that if $A_1$ and $A_2$ are ample and $A_1-A_2$ is ample, then $\ord_Z(\parallel D+A_1\parallel)\leq
\ord_Z(\parallel D+A_2\parallel)$. It is then easy to deduce that if $(A_m)_{m\geq 1}$
is a sequence of ample divisors whose classes in $\NS^1(X)_{\RR}$ converge to zero, then 
$\sigma_Z(D)=\lim_{m\to\infty}\ord_Z(\parallel D+A_m\parallel)$. Using the continuity of
$\ord_Z(\parallel-\parallel)$ on the big cone, we see that $\sigma_Z(D)=\ord_Z(\parallel D
\parallel)$ if $D$ is big.

 It is straightforward to see from definition that
$\sigma_Z(D)$ only depends on the equivalence class of $D$. Therefore we may and will
consider $\sigma_Z$ as a function on the pseudo-effective cone of $X$. 

\begin{proposition}\label{semicont}
The function $\sigma_Z$ is lower semi-continuous on the pseudo-effective cone.
\end{proposition}

\begin{proof}
Note that by Theorem~\ref{function_big_cone}, each function $\phi_A$ given by
$\phi_A(D)=\ord_Z(\parallel D+A\parallel)$ is continuous on the 
pseudo-effective cone (here $A$ is an arbitrary ample $\RR$-divisor). Since $\sigma_Z=\sup_A\phi_A$, it follows that
$\sigma_Z$ is lower semi-continuous.
\end{proof}

\begin{theorem}\label{pseudoeff}
Let $Z$ be an irreducible proper closed subset of $X$ and assume that either the ground field
$k$ is uncountable, or $Z$ consists of a point. If $D$ is a pseudo-effective divisor on $X$, then
the following are equivalent:
\begin{enumerate}
\item[i)] $Z$ is not contained in $\BB_-(D)$.
\item[ii)] $\sigma_Z(D)=0$.
\item[iii)] The ideal $\tau_+(\parallel mD\parallel)$ does not vanish along $Z$ for
any $m\geq 1$.
\end{enumerate}
\end{theorem}

\begin{proof}
Let us fix a sequence of ample $\RR$-divisors $(A_i)_{i\geq 1}$ whose classes
in $\NS^1(X)_{\RR}$ converge to zero, and such that all $D+A_i$ have rational coefficients. By definition, we have
$\sigma_Z(D)=0$ if and only if $\ord_Z(\parallel D+A_i\parallel)=0$ for all $i$.
On the other hand, Proposition~\ref{properties_nonnef} vi) gives 
$\BB_-(D)=\bigcup_i\BB_-(D+A_i)$, hence our hypothesis on $Z$ and $k$ implies
that $Z\not\subseteq\BB_-(D)$ if and only if for every $i$ we have $Z\not\subseteq
\BB_-(D+A_i)$. Therefore the equivalence of i) and ii) follows from the equivalence of  
i) and iv) in Theorem~\ref{char_nonnef} (see Remark~\ref{Q-divisor}). 

Suppose now that ii) holds, hence $\ord(\parallel D+A_i\parallel)=0$ for all $i$.
It follows from Theorem~\ref{char_nonnef} (see also Remark~\ref{Q-divisor}) that for every $m\geq 1$, the ideal $\tau(\parallel 
m(D+A_i)\parallel)$
does not vanish along $Z$. Since $\tau_+(\parallel mD\parallel)=\tau(\parallel m(D+A_i)
\parallel)$ for $i\gg 0$, we get the assertion in iii).

Suppose now that iii) holds. If ii) fails, then there is $i$ such that
$\ord_Z(\parallel D+A_i\parallel)>0$, hence for some $m\geq 1$, the ideal
$\tau(\parallel m(D+A_i)\parallel)$ vanishes along $Z$. Since we have
$\tau_+(\parallel mD\parallel)\subseteq\tau(\parallel m(D+A_i)\parallel)$, we obtain a contradiction with iii). This completes the proof of the theorem. 
\end{proof}

\begin{remark}
It is shown in \cite[Proposition~II.1.10]{Nak} that if $D$ is a pseudo-effective divisor on $X$ and $E_1,\ldots,E_r$ are prime divisors
such that $\sigma_{E_i}(D)>0$ for all $i$, then for every $\alpha_1,\ldots,\alpha_r\in
\RR_{\geq 0}$ and every $i$, one has $\sigma_{E_i}(\alpha_1E_1+\ldots+\alpha_rE_r)=
\alpha_i$ (note that the proof therein is characteristic-free). In particular, this implies that
the classes of $E_1,\ldots,E_r$ in $\NS^1(X)_{\RR}$ are linearly independent, hence
$r$ is bounded above by the Picard number $\rho(X)$ of $X$. If we assume that the ground field is uncountable, we deduce using Theorem~\ref{pseudoeff} that the number of irreducible codimension one subsets of $\BB_-(D)$ is bounded above by $\rho(X)$. 
\end{remark}

\providecommand{\bysame}{\leavevmode \hbox \o3em
{\hrulefill}\thinspace}


\begin{thebibliography}{BMS}

\bibitem[BMS]{BMS}
M.~Blickle, M.~Musta\c{t}\u{a}, and K.~E. Smith, Discreteness and rationality of $F$-thresholds,
Michigan Math. J. 57 (2008), 463--483.

\bibitem[ELMNP]{ELMNP}
 L. Ein, R. Lazarsfeld, M.
Musta\c{t}\v{a}, M. Nakamaye and M. Popa, Asymptotic
invariants of base loci,  Ann. Inst. Fourier (Grenoble) 56 (2006), 
1701--1734.

\bibitem[Fuj]{Fujita}
T.~Fujita, Semipositive line bundles, J. Fac. Sci. Univ. Tokyo Sect. IA Math. \textbf{30} (1983),  353--378. 

\bibitem[Ha]{Hara}
N.~Hara, A characteristic $p$ analog of multiplier ideals and applications, Comm. Algebra 
\textbf{33} (2005), 3375--3388. 

\bibitem[HT]{HT}
N.~Hara and S.~Takagi, On a generalization of test ideals,
 Nagoya Math. J. \textbf{175} (2004), 59--74.

\bibitem[HY]{HY}
N.~Hara and K.-i.~Yoshida,
 A generalization of tight closure and multiplier
 ideals, Trans. Amer. Math. Soc. 355 (2003),
 3143--3174.

\bibitem[JM]{JM}
M.~Jonsson and M.~Musta\c{t}\u{a}, Valuations and asymptotic invariants for sequences of ideals, Ann. Inst. Fourier, to appear.

\bibitem[Kee]{Keeler}
D.~S.~Keeler, Fujita's conjecture and Frobenius amplitude, Amer. J. Math. 130 (2008),
1327--1336.


\bibitem[Laz]{positivity}
R.~Lazarsfeld, \emph{Positivity in algebraic geometry} II, Ergebnisse der Mathematik und ihrer
Grenzgebiete, 3. Folge, Vol. 49, Springer-Verlag, Berlin, 2004.


\bibitem[Nak]{Nak}
N.~Nakayama, \emph{Zariski-decomposition and abundance}, MSJ Memoirs 14, Mathematical Society of Japan, Tokyo, 2004.

\bibitem[Sch]{Schwede}
K.~Schwede, A canonical linear system associated to adjoint divisors in characteristic $p > 0$,
preprint, arXiv:1107.3833. 

\bibitem[ST]{ST}
K.~Schwede and K.~Tucker, A survey of test ideals,
preprint, arXiv:1104.2000.



\end{thebibliography}
\end{document}